\documentclass[10pt,a4paper]{amsart}

\usepackage{amsmath,amssymb,graphicx}		
\usepackage[driver=dvips,text={372pt,533pt},centering]{geometry}

\numberwithin{equation}{section}
\usepackage{amssymb,color,amscd}
\usepackage{tikz}
\usepackage{lscape}
\usepackage{pdflscape}
\usepackage{pdfpages}
\usepackage{rotating}

\newtheorem{thm}{Theorem}[section]
\newtheorem{lemma}[thm]{Lemma}
\newtheorem{proposition}[thm]{Proposition}
\newtheorem{corollary}[thm]{Corollary}

\newtheorem{definition}[thm]{Definition}
\newtheorem{remark}[thm]{Remark}
\newtheorem{conjecture}[thm]{Conjecture}

\newtheorem{notation}[thm]{Notation} 
\newtheorem{convention}[thm]{Convention} 
\newtheorem{ex}[thm]{Example}

\newtheorem{claim}[thm]{Claim}
\newtheorem*{acknowledgments}{Acknowledgments}
\numberwithin{figure}{section}
\newcommand{\lr}[1]{\stackrel{\scriptsize{\raisebox{.5ex}[0pt][.5ex]{$#1$}}}{\rightarrow}}
 
 \newcommand{\dessin}[2]{
  \vcenter{\hbox{\includegraphics[height=#1]{#2.pdf}}}}

\begin{document} 
\title[Finite type invariants of welded string links and ribbon tubes]{On finite type invariants of welded string links and ribbon tubes} 
\author[A. Casejuane ]{Adrien Casejuane} 
\author[J.B. Meilhan ]{Jean-Baptiste Meilhan} 
\address{Univ. Grenoble Alpes, CNRS, IF, 38000 Grenoble, France}
\email{adrien.casejuane@univ-grenoble-alpes.fr}
\email{jean-baptiste.meilhan@univ-grenoble-alpes.fr}



\maketitle

\begin{abstract} 
Welded knotted objects are a combinatorial extension of knot theory, which can be used as a tool for studying ribbon surfaces in $4$-space. 
A finite type invariant theory for ribbon knotted surfaces was developed by Kanenobu, Habiro and Shima, and this paper proposes a study of these invariants, 
using welded objects. Specifically, we study welded string links up to $w_k$-equivalence, which is an equivalence relation  
introduced by Yasuhara and the second author in connection with finite type theory. 
In low degrees, we show that this relation characterizes the information contained by finite type invariants. 
We also study the algebraic structure of welded string links up to $w_k$-equivalence. 
All results have direct corollaries for ribbon knotted  surfaces. 
\end{abstract} 
\section{Introduction}
In the study of knotted surfaces, \emph{i.e.} smooth embeddings of $2$-dimensional manifolds in $4$-space, the class of ribbon surfaces has proved to be of particular interest. 
These are the analogue of ribbon knots in $3$-space, as defined by Fox in the sixties, 
in the sense that such knotted surfaces bound immersed $3$-manifolds with only one fixed topological type of so-called ribbon singularities.
Ribbon knotted surfaces were extensively studied from the early days higher dimensional knot theory, notably through the work of Yajima \cite{Yajima,Yaji} and Yanagawa \cite{yanagawa,yanagawa2,yanagawa3}. 
One nice feature of ribbon knotted surfaces is that they admit a natural notion of 'crossing change'.
Indeed, one can always arrange such surfaces so that the only crossings occur along circles of double points: swapping the over/under information along such a circle then yields a new ribbon surface. 
Just like the usual crossing change can be used to define Vassiliev knot invariants, Kanenobu, Habiro and Shima used this  local move to introduce a theory of finite type invariants for ribbon knotted surfaces \cite{Habiro-Kanenobu-Shima, KS}. For ribbon $2$-knots, Habiro and Shima further showed that finite type invariants are completely determined by the (normalized) Alexander polynomial.

This paper aims at characterizing, in a similar way, finite type invariants of ribbon tubes, which are ribbon knotted annuli in the $4$-ball whose boundary is given by fixed copies of the unlink. 
These objects were introduced in \cite{Audoux-Bellingeri-Meilhan-Wagner2018}, as a higher dimensional analogue of  string links.  
For $1$-component ribbon tubes, the situation is strictly the same as for ribbon $2$-knots, in the sense that all finite type invariants come from the coefficients $\alpha_k$ $(k\ge 2$) of the normalized Alexander polynomial \cite{Meilhan-Yasuhara}. For ribbon tubes of more components, a number of finite type invariants can be derived from this polynomial, as follows. Given a sequence $R$ of possibly overlined indices, there is a canonical procedure to connect the various components of a ribbon tube into a single annulus, and evaluating $\alpha_k$ on the latter yields a degree $k$ finite type invariant $\mathcal{I}_{R;k}$ of the initial ribbon tube, called a \emph{closure invariant}, see Definition \ref{def:closure}. 
Another family of finite type invariants of ribbon tubes is given by the higher-dimensional Milnor invariants defined in \cite{Audoux-Bellingeri-Meilhan-Wagner2018}.  
The first of these invariants is Milnor invariant $\mu(ij)$,  which in effect is  the (nonsymmetric) linking number of component $i$ with component $j$. The main result of this paper can be stated as follows.
\begin{thm}\label{thmain}
Let $T$ and $T'$ be two $n$-component ribbon tubes. The following are equivalent for $k \in \{2;3\}$. 
\begin{enumerate}
    \item $T$ and $T'$ are $RC_k$-equivalent. 
    \item For any finite type invariant $\nu$ of degree $< k$, we have $\nu(T) = \nu(T')$. 
    \item $T$ and $T'$ cannot be distinguished by the following invariants: 
    \begin{itemize}
		\item[$\bullet$ ] If $k=2$:  linking numbers $\mu(ij)$, for all $i\neq j$.
		\item[$\bullet$ ] If $k=3$:  linking numbers  $\mu(ij)$  for all $i\neq j$, and closure invariants 
		$\mathcal{I}_{(i) ; 2}$ for all $i$, $\mathcal{I}_{(j, i) ; 2}$, $\mathcal{I}_{(i, \overline{j}) ; 2}$ and  
			 $\mathcal{I}_{(\overline{i}, j) ; 2}$ for all $i<j$, and 
			 $\mathcal{I}_{(\overline{j}, i, \overline{k}) ; 2}$ for all pairwise distinct $i,j,k$ such that $j < k$.  
      \end{itemize}
\end{enumerate}
\end{thm}
Here, the $RC_k$-equivalence is an equivalence relation introduced by Watanabe in \cite{Watanabe} which, by the above, characterizes the information contained by finite type invariants of ribbon tubes of degree $<3$. 
This relation, discussed in Section \ref{sec:ribbon},  is an analogue in higher dimensions of Habiro's $C_k$-equivalence for usual knotted objects \cite{Habiro}. 
Watanabe showed that for ribbon $2$--knots, the $RC_k$-equivalence characterizes the information contained by all finite type invariants of degree $<k$ \cite[Thm.~1.1]{Watanabe}. 

Note that according to Theorem \ref{thmain}, any degree $2$ invariant of ribbon tubes can be expressed as a combination of linking numbers and closure invariants: we give such a formula for length $3$ Milnor invariants in Proposition \ref{mu_123}.
Note also that the case $k=2$ above was essentially already known by \cite{ABMW3}, see \cite[\S~7.2]{Meilhan-Yasuhara}.

In Section \ref{sec:A}, we investigate degree $3$ invariants in the $2$-component case. 
The classification result is only given modulo a conjectured relation, but it appears already at this stage that closure invariants no longer suffice to generate all degree $3$ invariants, since the classification also involves length $4$ Milnor invariants. 
See Remark \ref{rmqpourplustard}. 
An extensive study of degree $3$ invariants of ribbon tubes, and discussions on the perspectives that it opens, can be found in \cite{Casejuane}.
\medskip

The results of this paper on ribbon knotted surfaces are all obtained as consequences of diagrammatic results.
Another remarkable feature of ribbon surfaces is indeed that they can be described and studied using welded objects, which are a quotient of virtual knot theory, see Section \ref{sec:welded}. 
Early works of Yajima \cite{Yajima} highlighted the fact that all relations in the fundamental group of the complement of a ribbon $2$--knot can be  encoded using the usual diagrammatics of knot theory. This key observation was later completed and formalized by Satoh, as a surjective map from welded objects to ribbon knotted surfaces \cite{Satoh2000}. 
Hence the core of the present paper is a characterization of finite type invariants of welded string links. The theory of finite type invariants for welded objects is based on the virtualization move, which replaces a classical crossing by a virtual one, and the above-mentioned closure and Milnor invariants do have a strict analogue for welded string links, with the same finite type properties, see Section \ref{sec:inv}.. 

Our main tool will be the \emph{arrow calculus} developed in \cite{Meilhan-Yasuhara}, which is a welded analogue of Habiro's clasper calculus \cite{Habiro}. 
In particular, a family of finer and finer equivalence relations on welded objects called $w_k$-equivalence is defined in  \cite{Meilhan-Yasuhara}, which is closely related to finite type theory. It is indeed known that two $w_k$-equivalent welded objects cannot be distinguished by any finite type invariant of degree $<k$. 
The converse implication also holds for welded (long) knots, thus fully characterizing the information contained by finite type invariants of these objects, and it is conjectured that such an equivalence also holds for welded string links (Conjecture \ref{conjek}).
This is a natural analogue of the Goussarov-Habiro conjecture for finite type invariants of string links and homology cylinders \cite{Habiro} and, as a matter of fact, our results amount to verify this conjecture at low degree.
Indeed, our main diagrammatical results are classifications of welded string links up to $w_k$-equivalence for $k\leq 3$ by finite type invariants (Corollary \ref{corow3}), which imply Theorem \ref{thmain} as seen in Section \ref{sec:ribbon}. 
Several general results are also given on the set of $w_k$-equivalence classes of welded string links, showing in particular that these form a finitely generated, non abelian group. 

We conclude this introduction by mentioning the recent work of Colombari, who gave a complete classification of welded string links up to $w_k$-concordance, for all $k$, in \cite{colombari}. This equivalence relation, generated by $w_k$-equivalence and welded concordance, turns out to completely characterize welded string links (hence, classical string links) having the same Milnor invariants of length $\le k$. This result also shows that all  finite type concordance invariants of welded string links are given by Milnor invariants. 
\medskip

The rest of this paper is organized as follows. 
In Section \ref{sec:welded}, we review the notion of welded knotted objects and the basics of arrow calculus. 
Section \ref{sec:wkk} is devoted to the $w_k$-equivalence; some algebraic properties of welded string links up to $w_k$-equivalence are given in Section \ref{sec:grooop}.  Finite type invariants of welded string links are characterized at low degree in Section \ref{sec:SL3}. 
The topological counterparts of our results, including the proof of Theorem \ref{thmain}, are given in the final Section \ref{sec:ribbon}. 
\section{Welded objects and arrow calculus} \label{preliminaries}

\subsection{Welded knotted objects} \label{sec:welded}

Recall that a virtual diagram is a planar immersion of some $1$-dimensional manifold; the singular set is a finite collection of transverse double points endowed with a decoration, either as a classical or as a virtual crossing. In figures, classical crossings are represented as in usual knot diagrams, while virtual crossings are simply drawn as double points (we do not follow the customary convention using circled double points). 

\begin{definition}
A \emph{welded knotted object} is the equivalence class of a virtual diagram modulo the generalized Reidemeister moves and the OC move. 
Here the generalized Reidemeister moves consist of the three usual Reidemeister moves (involving classical crossings), 
and the Detour move shown in Figure \ref{fig:mixedOC}. The OC move is shown on the right-hand side of the same figure. 
\end{definition}
\begin{figure}[!h]
\begin{center}
 \includegraphics[scale=0.9]{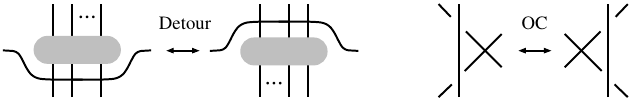} 
  \caption{The Detour and OC moves  
  {\footnotesize (In the Detour move, the grey part indicates any subdiagram, with classical and/or virtual crossings)}
  }\label{fig:mixedOC}
\end{center}
\end{figure}

Recall that usual (string) links inject into welded (string) links, in the sense that two diagrams without virtual crossings, that are related by a sequence of generalized Reidemeister and OC moves, represent isotopic objects, see \cite[Thm~1.B]{Goussarov-Polyak-Viro}. 
Welded objects are also intimately related to ribbon knotted surfaces in $4$-space, via Satoh's Tube map  \cite{Satoh2000}, as further developed in Section \ref{sec:ribbon}.

\begin{remark}
Virtual diagrams modulo generalized Reidemeister move give rise to virtual knotted objects, that were first introduced in the early nineties 
by L.~Kauffman in \cite{Kauffman1999}, and M.~Goussarov, M.~Polyak and O.~Viro in \cite{Goussarov-Polyak-Viro}. 
\end{remark}

This paper will mainly deal with the following class 
\begin{definition}
An $n$-component  \emph{welded string link} is the welded class of $n$ properly immersed copies of the unit interval into $[0,1] \times [0,1]$, endowed with $n$ fixed points on $[0,1] \times \{\varepsilon\}$ ($\varepsilon=0,1$), such that the $i$th copy of the interval runs from the $i$th fixed point in $[0,1] \times \{0\}$ to the $i$th fixed point in $[0,1] \times \{1\}$. A $1$-component welded string link is also called a \emph{welded long knot}. 
\end{definition}

We denote by $wSL(n)$ the set of welded string links. The stacking product endows $wSL(n)$ with a monoid structure, whose unit $\mathbf{1}$ is given by the trivial diagram of $n$ intervals, with no crossing. 

\subsection{Arrow calculus} \label{sec:calculflÃšches}

We now review the diagrammatic calculus for welded objects developed in \cite{Meilhan-Yasuhara}, called arrow calculus. This is a welded analogue of Habiro's clasper calculus for usual knotted objects \cite{Habiro} and, as such, it is intimately related to finite type invariants, see Section \ref{sec:typefini}. 
Let $L$ be some welded knotted object. 
\begin{definition}
A \emph{$w$-tree} for  $L$ is a connected unitrivalent tree $T$, immersed into the plane so that
\begin{itemize}
	\item[$-$] trivalent vertices are endowed with a cyclic order, are pairwise disjoint and disjoint from $L$; 
	\item[$-$] univalent vertices are pairwise disjoint and lie in  $L \setminus \{\textrm{crossings of $L$}\}$; 
	\item[$-$] edges are oriented, so that each trivalent vertex involves exactly one outgoing edge; 
	\item[$-$] edges may contain virtual (but not classical) crossings, either with $L$ or with $T$ itself;
	\item[$-$] edges may contain decorations $\bullet$, called \emph{twists}, 
	                which are subject to the involutive rule that two consecutive twists do cancel. 
\end{itemize}
Moreover, for a union of $w$Ðtrees for $L$, we assume that vertices are pairwise disjoint, and that
crossings among edges are all virtual.
\end{definition}

\noindent We shall call \emph{tails} and \emph{head} the endpoints of a $w$-tree, according to the orientation. 

\begin{definition}
The \emph{degree} of a $w$-tree is the number of tails. For $k \ge 1$, we call \emph{$w_k$-tree} a $w$-tree of degree $k$. 
\end{definition}

Now, a $w$-tree is an instruction for modifying $L$, according to a process which we abusively call surgery, defined as follows. 

\begin{definition}
Let $A$ be a $w_1$-tree for the diagram $L$. The \emph{surgery on $L$ along $A$} yields a new welded diagram $L_A$ according to the local rule: 
$$ \textrm{\includegraphics{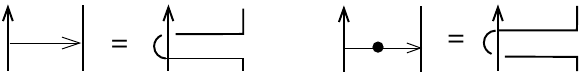}. } $$
If $A$ crosses (virtually) either $L$ or some other $w$-tree, then the strands of $L_A$ likewise cross the same object virtually.

In general, if $T$ is a $w_k$-tree for $L$, then surgery along $T$ is defined as surgery along the union of $w_1$-trees $E(T)$, called the \emph{expansion} of $T$ and defined recursively by the rule:  
$$ \textrm{\includegraphics[scale=0.9]{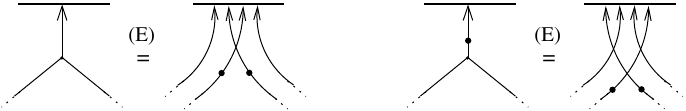}.} $$
\end{definition}
\begin{remark}
In the above figure, the dotted parts represent parallel subtrees, which are parallel copies of the non-depicted part ot the initial $w$-tree, that always cross each other virtually -- see \cite[Conv.~5.1]{Meilhan-Yasuhara} for a detailed explanation. 
\end{remark}

A key point is that any welded object $L$ can be represented in this way as a union of some diagram \emph{with no classical crossing} and some $w$-trees, called an arrow presentation for $L$. 
Since we shall be concerned with welded string links in this paper, let us define this notion more formally in this particular context.

\begin{definition}
Let $L$ be an $n$-component welded string link.
An \emph{arrow presentation} $(\mathbf{1}, T)$ of $L$ consists of the trivial diagram $\mathbf{1}$,  together with a union of $w$-trees $T$ for  $\mathbf{1}$ such that $L = \mathbf{1}_T$. 
Two arrow presentations are \emph{equivalent} if they represent equivalent welded diagrams. 
\end{definition}

By \cite[Prop.~4.2]{Meilhan-Yasuhara}, any element of $wSL(n)$ admits an arrow presentation; moreover, a complete set of relations is known, that relates any two arrow presentations of a given diagram: 
\begin{thm} \cite[Thm.~5.21]{Meilhan-Yasuhara} \label{Mouvements de flÃšches}
Two arrow presentations are equivalent if and only if they are related by a sequence of the following moves: 
\begin{enumerate}
	\item[(1)] Any generalized Reidemeister move involving $w$-trees and/or the diagram, and the following local moves:
	\begin{center}
		\includegraphics[scale=0.83]{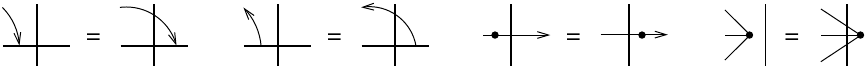} 
		\end{center}
	\item[(2)] Head and Tail reversal:\\[-0.5cm]
		\begin{center}
		\includegraphics[scale=0.9]{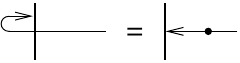} \qquad \qquad
		\includegraphics[scale=0.9]{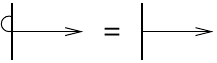} 
		\end{center}
	\item[(3)] Tails exchange (tails may or may not belong to the same $w$-tree):\\[-0.3cm]
		\begin{center}
		\includegraphics[scale=1]{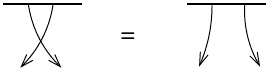} 
		\end{center}
	\item[(4)] Isolated arrow:\\[-0.5cm]
		\begin{center}
		 \includegraphics[scale=0.8]{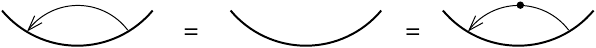} 
		\end{center}
		\item[(5)] Inversion:\\[-0.5cm]
		\begin{center}
		\includegraphics[scale=0.9]{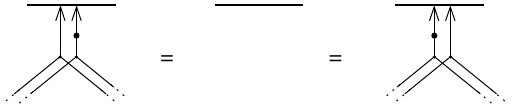}
		\end{center}
	\item[(6)] Slide:\\[-0.5cm]
		\begin{center}
		\includegraphics[scale=0.9]{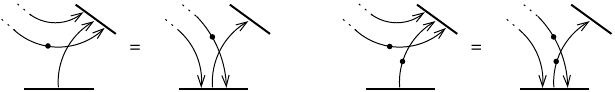} 
		\end{center}

\end{enumerate}
\end{thm}

Moreover, a collection of further operations on arrow presentations can be derived from these moves, as summarized below. 
\begin{proposition} \cite[Lemmas 5.14 to 5.18]{Meilhan-Yasuhara} \label{Ãchanges}
The following local moves give equivalent arrow presentations. 
\begin{enumerate}
	\item[(7)] Heads exchange:\\[-0.5cm]
		\begin{center}
		\includegraphics[scale=0.85]{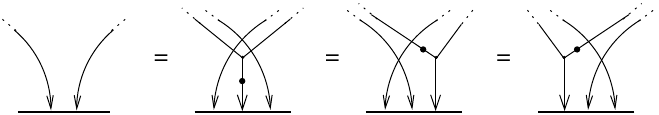} 
		\end{center}
	\item[(8)] Head/Tail exchange:\\[-0.5cm]
		\begin{center}
		\includegraphics[scale=0.85]{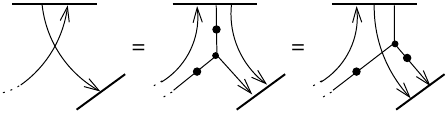} 
		\end{center}
	\item [(9)] AS (Antisymmetry):\\[-0.5cm]
		\begin{center}
		\includegraphics[scale=0.85]{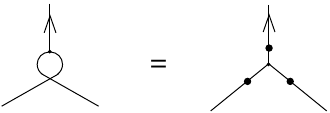}
		\end{center}
	\item[(10)] Fork:\\[-0.5cm]
		\begin{center}
		 \includegraphics[scale=0.7]{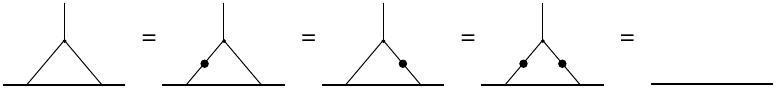}
		 \end{center}
\end{enumerate}
\end{proposition}

\begin{convention}
In what follows, we will blur the distinction between arrow presentations and the welded diagrams obtained by surgery.  
Moreover, we shall use the moves of the previous two results by only referring to their numbering (1)-(10). 
\end{convention}

\subsection{Invariants of welded string links}\label{sec:inv}

Recall that given a welded string link $L$, there is an associated \emph{welded group} $G(L)$, which is abstractly generated by all arcs and with a conjugating relation associated with each classical crossing, see e.g. \cite[\S~6.1]{Meilhan-Yasuhara}. 
We stress that, if $L$ is a classical string link, then $G(L)$ is the fundamental group of the complement. 

\subsubsection{Closure invariants}

We now introduce a family of invariants of welded string links, defined by evaluating the normalized Alexander polynomial on some welded long knot built via a closure process. 

Let $K$ be a  welded long knot.
The \emph{normalized Alexander polynomial} $\tilde{\Delta}_K(t)\in \mathbb{Z}[t^{\pm 1}]$ of $K$ 
was first defined in  \cite{Habiro-Kanenobu-Shima}; see \cite[\S~6.2]{Meilhan-Yasuhara} for a review. 
\begin{definition}
The \emph{$k$th normalized coefficient} of the Alexander polynomial is the coefficient $\alpha_k(K)$ in the power series expansion of  $\tilde{\Delta}_K(t)$ at $t=1$: 
$$\tilde{\Delta}_K(t) = 1 + \sum_{k=2}^{\infty} \alpha_k(K) (1-t)^k.$$
\end{definition}

We now proceed with defining the general closure process underlying closure invariants.
\begin{definition}
Let $n \in \mathbb{N}$. 
A \emph{list} of length $k$ ($k < n$) is a sequence of pairwise distinct, possibly overlined integers in $\{1,\cdots,n\}$.
\end{definition}
A list is an instruction for closing an $n$-component welded string link into a welded long knot. 
\begin{definition}
Let $L$ be an $n$-component welded string link, and let $R$ be a list of length $< n$.
Then $Cl_R(L)$ is the welded long knot obtained as follows:  
\begin{itemize}
	\item delete a neighborhood $L\cap ([0,1]\times [0,\varepsilon[\cup [0,1]\times ]1-\varepsilon,1])$ of the boundary of $L$, and fix the points $p_i:=\{\frac{1}{2}\}\times \{i\}$ ($i=0,1$) on the boundary of $[0,1]\times [0,1]$;  
	\item delete all components whose index does not appear in $R$; 
	\item reverse the orientation of all components whose index is overlined in $R$; 
	\item build a welded long knot, starting at $p_0$ and ending at $p_1$, by connecting these $k$ oriented strands endpoints following the order of the list $R$ and the orientation of each strand, with arbitrary arcs that cross virtually the rest of the diagram. 
\end{itemize}
\end{definition}
See Figure \ref{fig:exclose} for a couple examples. 
Observe that this process is well-defined thanks to the Detour move. 
Note also that the process extends naturally to arrow presentations, by closing the trivial diagram $\mathbf{1}$ as instructed by the list $R$. 
\begin{ex}\label{exclose}
Consider the arrow presentation for $L\in wSL(2)$ shown in the middle of Figure \ref{fig:exclose}.  
\begin{figure}[!h]
\begin{center}
 \includegraphics[scale=1]{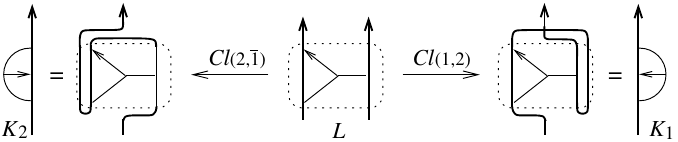} 
  \caption{The closures $Cl_{(2,\overline{1})}(L)$ (left) and $Cl_{(1,2)}(L)$ (right) of the welded string link $L$.
  }\label{fig:exclose}
\end{center}
\end{figure}
On the one hand, the closure $Cl_{(1, 2)}(L)$ is the welded long knot $K_1$ represented on the right-hand side of the figure. On the other hand, as shown in the same figure, the closure $Cl_{(2, \overline{1})}(L)$  gives the welded long knot $K_2$. 
Furthermore, we have that $Cl_{(1, \overline{2})}(L)=K_1$ and $Cl_{(\overline{2}, \overline{1})}(L)=K_2$.
Finally, by the Fork move (10), $Cl_{R}(L)$ is the trivial long knot for 
$R=(\overline{1},2), (\overline{1},\overline{2}), (2,1), (\overline{2},1)$.
\end{ex}

We can now define closure invariants of welded string links. 
\begin{definition}\label{def:closure}
Let $R$ be a list, and let $k\ge 2$ be some integer. 
The \emph{closure invariant} $\mathcal{I}_{R ; k}$ is the welded string link invariant defined by  $\mathcal{I}_{R ; k}(L) = \alpha_k(Cl_R(L))$.
\end{definition}
In particular, the closure invariant $\mathcal{I}_{(i);k}$ simply computes the normalized Alexander coefficient $\alpha_k$ of the $i$th component of a welded string link. 

It is straightforwardly checked that closure invariants indeed are invariants of welded string links. 
This family of invariants should be compared to the closure invariants of classical string links of \cite{MY3}, and further developed in \cite[\S~5.1]{MY_GT}; note however that in these latter works, the closure operations introduce classical crossings.

\subsubsection{Welded linking numbers and Milnor invariants}

Given an $n$-component welded string link $L$ and two distinct indices $i,j\in\{1,\cdots,n\}$, the \emph{welded linking number} $\mu(ij)$ is given by
$$  \mu_L(ij) = \sum_{c\in C_{i,j}} \textrm{sign}(c), $$
where the sum runs over the set $C_{i,j}$ of classical crossings where component $i$ passes \emph{over} component $j$, and where the sign of the crossing is given by the usual rule: 
\begin{center}
\begin{tikzpicture}
	\draw[<-] (0,1) -- (0.4,0.6);
	\draw[->] (0,0) -- (1,1);
	\draw (0,0) node [left] {$i$};
	\draw (0.6,0.4) -- (1,0);
	\draw (1,0) node [right] {$j$};
	\draw (0.5,0) node [below] {$+1$};
	
	\draw[<-] (3,1) -- (4,0);
	\draw (4,0) node [right] {$i$};
	\draw (3,0) -- (3.4,0.4);
	\draw (3,0) node [left] {$j$};
	\draw[->] (3.6,0.6) -- (4,1);
	\draw (3.5,0) node [below] {$-1$};
\end{tikzpicture}
\end{center}
It is quite straightforward to verify that this indeed defines a welded invariant. If $L$ is a classical string link, then clearly we have that  $\mu(ij)=\mu(ji)$ is the usual linking number. 

These invariants were first introduced in \cite{Goussarov-Polyak-Viro}, under the name of virtual linking numbers. 
Just like usual linking numbers were widely generalized into Milnor invariants in \cite{Milnor}, 
there is a welded extension of Milnor invariants $\mu(I)$ for any sequence of indices $I$, which generalizes the welded linking numbers. This extension was  first given in  \cite[Sec.~6]{Audoux-Bellingeri-Meilhan-Wagner2018} using a topological approach and the Tube map (see Section \ref{sec:ribbon}), and a purely diagrammatic version was later provided in \cite{Miyazawa-Wada-Yasuhara}. 

\subsubsection{Finite type invariants}\label{sec:typefini}

We now recall the definition of finite type invariants of welded objects, and observe that the above invariants all fall into this category. 

Recall that a \emph{virtualization move} on a welded diagram is the replacement of a classical crossing by a virtual one. 
Given a welded diagram $L$, and a subset $S$ of the set of classical crossings of $L$, we denote by $L_S$ the welded diagram obtained by applying the virtualization move to all crossings in $S$. 

\begin{definition}\label{def:fti}
Let $\nu$ be an invariant of welded string links, taking values in some abelian group. 
Then $\nu$ is a \emph{finite type invariant of degree $\le k$} if, for any $L\in wSL(n)$ and any set $S$ of $k+1$ classical crossings in $L$, we have 
  $$\sum_{S' \subset S} (-1)^{\vert S' \vert} \nu(L_{S'}) = 0.$$
This is a finite type invariant of degree $k$ if, moreover, it is not of degree $\le k-1$.
\end{definition}

This definition was first given in  \cite[Sec.~2.3]{Goussarov-Polyak-Viro} in the context of virtual knots and links. 
Actually, in that same paper, the authors further identified the first nontrivial invariants of the theory: 
\begin{lemma}\cite{Goussarov-Polyak-Viro} \label{Milnortypefini}
For all $i, j \in \{1,\cdots,n\}$, the welded linking number $\mu(ij)$ is a degree $1$ finite type invariant of welded string links.
\end{lemma}
There are finite type invariants in any degree. Indeed, we have the following. 
\begin{lemma}\cite{Habiro-Kanenobu-Shima}\label{Alextypefini}
For all $k \geq 2$, $\alpha_k$ is a degree $k$ finite type invariant of welded long knots.
\end{lemma}
As an immediate consequence, we have: 
\begin{corollary}\label{cor:fticlose}
For all $k \geq 2$ and all list $R$,  the closure invariant $I_{R, k}$ is a degree $k$ finite type invariant of welded string links.
\end{corollary}
\begin{remark}
We note that Lemma \ref{Milnortypefini} generalizes to all Milnor invariants, in the sense that for a sequence $I$ of $k \geq 2$ indices in $\{1,\cdots,n\}$, the welded Milnor invariant $\mu(I)$  is a degree $k-1$ finite type invariant of welded string links.
A complete proof of this fact can be found in the Appendix of \cite{Casejuane}.
\end{remark}

\section{$w_k$-equivalence}\label{sec:wkk}

In this section, we review the family of equivalence relations introduced in \cite{Meilhan-Yasuhara}, called $w_k$-equivalence, and recall how it can be used as a tool for studying finite type invariants. Although the definition can be made in the general context of welded objects, we shall restrict ourselves below to welded string links; in particular, we investigate the algebraic properties of the group of $w_k$-equivalence classes of welded string links.

\subsection{Definition and relation to finite type invariants}

Let $k$ be a positive integer. 
\begin{definition}
Two welded string links $L,L'$ are \emph{$w_k$-equivalent}, denoted by $L\stackrel{k}{\sim} L'$, if there exists a finite sequence $(L_i)_{0 \leq i \leq n}$ of elements of $wSL(n)$ such that  $L_0 = L$, $L_n = L'$ and for each $i$, $L_{i+1}$ is obtained from $L_i$ either by surgery along a $w_l$-tree for some $l \geq k$ or by a generalized Reidemeister or OC move.
\end{definition}
Using the Expansion move (E), one sees that  the $w_k$-equivalence becomes finer as $k$ increases. 
This notion turns out to be closely related to finite type theory. 
\begin{proposition} \cite[Prop.~7.5]{Meilhan-Yasuhara} \label{thm:wkFTI}
For $k \geq 2$, two welded (string) links that are $w_k$-equivalent, share all finite type invariants of degree $< k$. 
\end{proposition}
Furthermore, it is proved in \cite{Meilhan-Yasuhara} that the converse holds for welded knots and welded long knots. 
For welded string links, we are naturally led to the following, which was first discussed in \cite[\S~10.3]{Meilhan-Yasuhara}.  
\begin{conjecture}\label{conjek}
Two welded string links are $w_k$-equivalents if and only if they cannot be distinguished by finite type invariants of degree $<k$. 
\end{conjecture}
\noindent This can be seen as a welded analogue of the Goussarov-Habiro conjecture, see \cite{Habiro}. 
As a matter of fact, Corollary \ref{corow3} validates this conjecture at low degree.

\subsection{Refined arrow calculus}

When working up to $w_k$-equivalence, the arrow calculus can be further refined: the point is that working up to $w_k$-equivalence allows for operations 'up to higher order terms'.  Indeed, in addition to the ten moves of Theorem \ref{Mouvements de flÃšches} and Proposition \ref{Ãchanges}, we have a number of extra operations at our disposal for manipulating arrow presentations.
Some of these operations are summarized in Lemma \ref{lem:techno} below, whose proof can be found in \cite[Sec.~7.4]{Meilhan-Yasuhara}. 
They are given in terms that are slightly stronger than $w_k$-equivalence, as follows.
Given two arrow presentations $(\mathbf{1},T)$ and $(\mathbf{1},T')$ and some integer $k \ge 1$, we denote by\\[-0.3cm]
$$ (\mathbf{1},T) \lr{k} (\mathbf{1},T')$$
the fact that  $(\mathbf{1},T) = (\mathbf{1},T'\cup T'')$ for some union $T''$ of w-trees of degree $\ge k$.
Note that $(\mathbf{1},T) \lr{k} (\mathbf{1},T')$ implies that $\mathbf{1}_T \stackrel{k}{\sim} \mathbf{1}_{T'}$. 

\begin{lemma}\label{lem:techno} 
Let $k, k'$ be integers.  
\begin{enumerate}
	\item[(11)] Twist: If $k\ge 2$, for any $w_k$-tree containing a twist we have\\[-0.5cm]
		\begin{center}
		\includegraphics[scale=0.9]{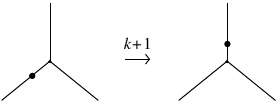} 
		\end{center}
	\item[(12)] Generalized Head/Tail exchange: We have \\[-0.5cm]
		\begin{center}
		\includegraphics[scale=0.85]{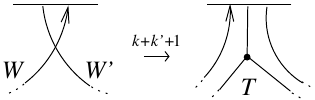} 
		\end{center}
        where $W$ and $W'$ are $w$-trees of degree $k$ and $k'$, respectively, so that $T$ is a $w_{k+k'}$-tree.  
        \item[(13)] IHX: If $k\ge 3$ we have\\[-0.5cm]
                \begin{center}
		\includegraphics[scale=0.85]{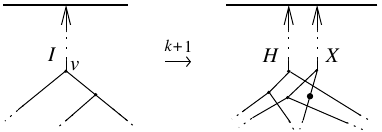} 
		\end{center}       
	where $I$, $H$ and $X$ are three $w_k$-trees as shown.  
\end{enumerate}
\end{lemma}
\noindent As with arrow moves (1)-(10), we shall use the relations of Lemma \ref{lem:techno} by only referring to their numbering. 

\begin{remark}
 Combining the Twist relation (11) with the with the reversal move (2) and the AS move (9), 
 we have that the two welded long knots $K_1$ and $K_2$ of Example \ref{exclose}, are $w_3$-equivalent: 
 \begin{center}
\begin{tikzpicture}
	\draw[<-] [very thick] (0,2) -- (0,0);
	\draw (0,1.5) arc(90:270:0.5);
	\draw[<-] (0,1) -- (-0.5,1);
	
	\draw (0.5,1) node {$=$};
	\draw[<-] [very thick] (1.5,2) -- (1.5,0);
	\draw (1.5,1.5) arc(90:270:0.5);
	\draw (1,1) -- (1.5,1);
	\draw (1.25,0.975) node {$\bullet$};
	\draw[->] (1.5,1) arc(270:450:0.125);
	
	\draw (2,1) node {$=$};
	\draw[<-] [very thick] (2.5,2) -- (2.5,0);
	\draw[<-] (2.5,1) -- (3,1);
	\draw (2.75,0.975) node {$\bullet$};
	\draw (3,1) arc(180:0:0.125);
	\draw (3.25,1) arc(0:-90:0.75);
	\draw (3,1) arc(180:360:0.125);
	\draw (3.25,1) arc(0:90:0.75);
	
	\draw (3.5,1) node {$=$};
	\draw[<-] [very thick] (4,2) -- (4,0);
	\draw (4,1.5) arc(90:-90:0.5);
	\draw (4.25,1.4) node {$\bullet$};
	\draw (4.25,0.6) node {$\bullet$};
	\draw[<-] (4,1) -- (4.5,1);
	
	\draw (5,1) node {$\stackrel{3}{\sim}$};
	\draw[<-] [very thick] (5.5,2) -- (5.5,0);
	\draw (5.5,1.5) arc(90:-90:0.5);
	\draw[<-] (5.5,1) -- (6,1);
\end{tikzpicture}
\end{center}
\end{remark}

Combining relation (12) with the previous exchange moves (3) and (7), we have the following.
\begin{corollary} \cite[Cor.~7.13]{Meilhan-Yasuhara} \label{lem:exchange}
Let $T$ and $T'$ be $w$-trees of degree $k$ and $k'$, respectively. One can freely exchange the relative position of two adjacent univalent vertices (head or tail) of $T$ and $T'$ at the cost of extra $w$-trees, all of degree $\ge k+k'$. 
\end{corollary}

This can be used to rearrange any arrow presentation of an element of $wSL(n)$ into a product of elementary pieces, each obtained by surgery along a single $w$-tree, as in \cite[Lem.~7.15]{Meilhan-Yasuhara}. 

Another noteworthy consequence of Corollary \ref{lem:exchange} is the following additivity property for closure invariants. 
\begin{proposition} \label{clÃŽture et additivitÃ©}
Let $k$ and $k'$ be two integers, and let $R$ be a list.
Let $W$ and $W'$ be unions of $w$-trees for the trivial diagram $\mathbf{1}$, of degree $\ge k$ and $\ge k'$, respectively.
For all $d < k+k'$, we have
$$\mathcal{I}_{R ; d}(\mathbf{1}_{W}\cdot \mathbf{1}_{W'}) = \mathcal{I}_{R ; d}(\mathbf{1}_{W}) + \mathcal{I}_{R ; d}(\mathbf{1}_{W}).$$
\end{proposition}
\begin{proof}
It is an immediate consequence of Corollary \ref{lem:exchange} that 
$Cl_R(\mathbf{1}_{W})\cdot Cl_R(\mathbf{1}_{W'}) \stackrel{k+k'}{\sim} Cl_R(\mathbf{1}_{W}\mathbf{1}_{W'})$. Now, since $\alpha_d$ is a finite type invariant of degree $d$ by Proposition \ref{Alextypefini}, we have $\alpha_d(Cl_R(\mathbf{1}_{W})\cdot Cl_R(\mathbf{1}_{W'}))=\alpha_d(Cl_R(\mathbf{1}_{W}\mathbf{1}_{W'}))$, and the result follows from the additivity property of $\alpha_d$ as given in \cite[Cor.~6.6]{Meilhan-Yasuhara}.
\end{proof}

\subsection{The group of welded string links up to $w_k$-equivalence}\label{sec:grooop}

Let $n, k \in \mathbb{N}^\ast$. 
We denote by $wSL(n)_k$ the set of $w_k$-equivalence classes of $n$-component welded string links. 

As already observed in \cite[\S~7.2]{Meilhan-Yasuhara}, $wSL(n)_1$ is the trivial group for all $n\ge 1$. 
It is also known that $wSL(1)_k$ is a finitely generated abelian group for all $k\ge 1$ \cite[Cor.~8.8]{Meilhan-Yasuhara}. 
In the general case, we have the following results.

\begin{thm}\label{thm:fgg}
For $n, k \in \mathbb{N}^\ast$,  $wSL(n)_k$ is a finitely generated group.
\end{thm}
\begin{proof}
Let us first prove the group structure. 
Let $F$ be a union of $w$-trees of degree $\ge l$  for $\mathbf{1}$ ($l\le k$). 
Consider a union $F'$ of $w$-trees, which consists of a parallel copy of each $w$-tree in $F$, that only differs by a twist $\bullet$ next to the head. 
By the Inversion move (5), we have that $(\mathbf{1},F\cup F')$ is equivalent to $(\mathbf{1},\emptyset)$. 
Now, we can use Corollary \ref{lem:exchange} to move $F'$ above a disk $D$ containing $F$; this introduces a union of $w$-trees $\tilde W$ of degree $\ge 2l$, which may intersect $D$. Next $\tilde W$ can in turn be moved above the disk $D$ using Corollary \ref{lem:exchange}, and this introduces another union of $w$-trees, each of degree $\ge 3l$. 
Iterating this process, we eventually obtain in this way that the trivial diagram $\mathbf{1}$ is $w_k$-equivalent to a product $\mathbf{1}_F\cdot \mathbf{1}_{W}$, where $W$ is a union of $w$-trees, disjoint from $D$. We have thus built the inverse of $\mathbf{1}_F$ up to $w_k$-equivalence.
Now, it remains to observe that the group  $wSL(n)_k$ is finitely generated, since there are only finitely many $w$-trees in each finite degree. 
\end{proof}

\begin{remark} \label{identification d'inverse}
A consequence of this proof, in the case $k = l+1$, is the following. 
If $T$ is a $w_l$-tree for $\mathbf{1}$, then $\mathbf{1}_T\cdot \mathbf{1}_{T^{\bullet}} \stackrel{{l+1}}{\sim} \mathbf{1}$, 
where $T^{\bullet}$ is obtained by inserting a twist $\bullet$ near the head of $T$.
\end{remark}

\begin{proposition}\label{lem:nonab}
The group $wSL(n)_k$ is not abelian for $n\ge 2$ and $k\ge 2$.
\end{proposition}
\begin{proof}
Consider the $2$-component welded string links $D$ and $D'$ shown on the left-hand side of Figure \ref{fig:nonab}. 
\begin{figure}[!h]
\begin{center}
 \includegraphics[scale=0.7]{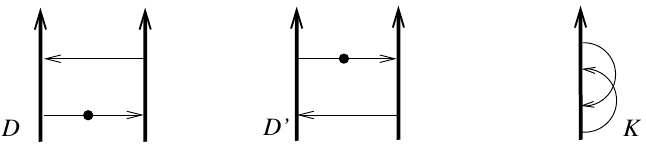} 
 \caption{The welded string links $D$ and $D'$, and the welded long knot $K$} \label{fig:nonab}
\end{center}
\end{figure}
On one hand,  by the Reversal moves (2), the closure $Cl_{(1,2)}(D)$ is the welded long knot $K$ shown on the right-hand side of the figure. 
On the other hand, by the Tails exchange and Isolated arrow moves (3) and (4), the welded long knot $Cl_{(1,2)}(D')$ is trivial. It is easily checked that $\alpha_2(K)=1$, thus proving that the closure invariant $\mathcal{I}_{(1,2);2}$ distinguishes $D$ and $D'$. By Proposition \ref{Alextypefini}, this proves that $D$ and $D'$ are not $w_2$-equivalent, hence not $w_k$-equivalent for any $k\ge 2$. 
\end{proof}

\section{Characterization of low degree invariants of welded string links}\label{sec:SL3}

The main results of this paper follow from a complete description of the group $wSL(n)_k$ for low values of $k$. 

\subsection{Classification of welded string links up to $w_3$-equivalence}

For pairwise distinct indices $i,j,k$, let $Z_{i,j}$, $E_i$, and $G_{i, j, k}$  be the following welded string links:  
\begin{center}
\begin{tikzpicture}
	\draw[<-] [very thick] (3,2) -- (3,0);
	\draw (2.75,0.25) node [below] {$i$};
	\draw (3,1) -- (4,1);
	\draw[->] (4,1) arc(270:450:0.125);
	\draw[<-] [very thick] (4,2) -- (4,0);
	\draw (4.25,0.25) node [below] {$j$};
	\draw (3.5,0) node [below] {$Z_{i, j}$};
\end{tikzpicture}
$\qquad$
\begin{tikzpicture}
	\draw[<-] [very thick] (6,2) -- (6,0);
	\draw (5.75,0.25) node [below] {$i$};
	\draw[<-] (6,1) -- (6.5,1);
	\draw (6,0.5) arc(270:450:0.5);
	\draw (6.1,0) node [below] {$E_i$};
\end{tikzpicture}
$\qquad$
\begin{tikzpicture}	
	\draw[<-] [very thick] (6,2) -- (6,0);
	\draw (5.75,0.25) node [below] {$i$};
	\draw[<-] [very thick] (7,2) -- (7,0);
	\draw (6.75,0.5) node [below] {$j$};
	\draw[<-] [very thick] (8,2) -- (8,0);
	\draw (8.25,0.25) node [below] {$k$};
	\draw[<-] (6,1.5) -- (7.5,1.5);
	\draw (7.5,1.5) -- (7.5,0.5);
	\draw (7,0.5) -- (8,0.5);
	\draw (7,0) node [below] {$G_{i, j, k}$};
\end{tikzpicture}
\end{center}
Furthermore, for all pair of indices $i,j$ such that $i<j$, let $A_{i,j}$, $B_{i,j}$, $C_{i,j}$ and $D_{i,j}$ be the following welded string links:  
\begin{center}
\begin{tikzpicture}
	\draw[<-] [very thick] (0,2) -- (0,0);
	\draw (-0.25,0.25) node [below] {$i$};
	\draw[<-] (0,1.5) -- (0.5,1);
	\draw (0,0.5) -- (0.5,1);
	\draw (0.5,1) -- (1,1);
	\draw[<-] [very thick] (1,2) -- (1,0);
	\draw (1.25,0.25) node [below] {$j$};
	\draw (0.5,0) node [below] {$A_{i, j}$};
\end{tikzpicture}
$\qquad$
\begin{tikzpicture}
	\draw[<-] [very thick] (0,2) -- (0,0);
	\draw (-0.25,0.25) node [below] {$i$};
	\draw (0,1.5) -- (0.5,1);
	\draw[<-] (0,0.5) -- (0.5,1);
	\draw (0.5,1) -- (1,1);
	\draw[<-] [very thick] (1,2) -- (1,0);
	\draw (1.25,0.25) node [below] {$j$};
	\draw (0.5,0) node [below] {$B_{i, j}$};
\end{tikzpicture}
$\qquad$
\begin{tikzpicture}
	\draw[<-] [very thick] (3,2) -- (3,0);
	\draw (2.75,0.25) node [below] {$i$};
	\draw (3,1) -- (3.5,1);
	\draw (3.5,1) -- (4,1.5);
	\draw[->] (4,1.5) arc(270:450:0.125);
	\draw (3.5,1) -- (4,0.5);
	\draw[<-] [very thick] (4,2) -- (4,0);
	\draw (4.25,0.25) node [below] {$j$};
	\draw (3.5,0) node [below] {$C_{i, j}$};
\end{tikzpicture}
$\qquad$
\begin{tikzpicture}	
	\draw[<-] [very thick] (3,2) -- (3,0);
	\draw (2.75,0.25) node [below] {$i$};
	\draw (3,1) -- (3.5,1);
	\draw (3.5,1) -- (4,1.5);
	\draw (3.5,1) -- (4,0.5);
	\draw[->] (4,0.5) arc(90:-90:0.125);
	\draw[<-] [very thick] (4,2) -- (4,0);
	\draw (4.25,0.25) node [below] {$j$};
	\draw (3.5,0) node [below] {$D_{i, j}$};
	 \draw (5.5,0.75) node [above] {$(i<j)$};
\end{tikzpicture}
\end{center}
For each of the above elements $X$ of $wSL(n)$, we also denote by $X^{-1}$ the welded string link obtained by inserting a twist $\bullet$ near the head of the defining $w$-tree.
For $n\ge 0$, we also denote by $X^{-n}$ the product of $n$ copies of $X^{-1}$. 

We note the following, rather non-intuitive relation. 
\begin{proposition}\label{relationABCD}
For all $i, j \in \{1,\cdots,n\}$ such that $i<j$, we have $$A_{i, j}\cdot B_{i, j}\cdot C_{i, j}\cdot D_{i, j} \stackrel{3}{\sim}  \mathbf{1}.$$
\end{proposition}
\begin{proof}
Starting with the union $U$ of four $w_1$-trees shown on the left-hand side below, the result follows from a sequence of Head/Tail exchange moves (8). In the figures, we indicate by an $\ast$ the place where each such move is performed. 
On one hand, we have the following sequence of equivalences: 
 $$ 
    U:\ \dessin{1.9cm}{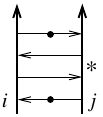} \ \stackrel{3}{\sim} \ \dessin{1.9cm}{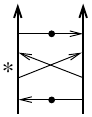}\ \cdot \ A_{i,j}^{-1}\stackrel{3}{\sim} \  \dessin{1.9cm}{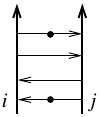}\ \cdot \ A_{i,j}^{-1}\cdot C_{i,j}^{-1} \ \stackrel{3}{\sim} \ A_{i,j}^{-1}\cdot C_{i,j}^{-1}. 
 $$
Here, the first equivalence is given by the Head/Tail exchange move (8) on the $j$th component of $\mathbf{1}$.  The $w_2$-tree created by this move  is a copy of $A_{i,j}^{-1}$, by the Twist move (11), which can be isolated from the rest of the diagram using Corollary \ref{lem:exchange}. 
The second equivalence above is given by the Head/Tail exchange move (8) on the $i$th component. 
This introduces a $w_2$-tree which is a copy of $C_{i,j}^{-1}$ by the Head reversal move (2) and the Twist move (11), and which can  also be isolated by Corollary \ref{lem:exchange}. The third equivalence then follows directly from the Inversion move (5). 
On the other hand, starting with the same union $U$ of $w_1$-trees, one can perform a similar sequence of Head/Tail exchange moves, but in the opposite order. This gives the following sequence of equivalences: 
 $$ 
    U:\dessin{1.9cm}{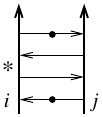} \ \stackrel{3}{\sim} \ \dessin{1.9cm}{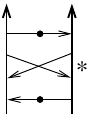}\ \cdot \ B_{i,j}\stackrel{3}{\sim} \  B_{i,j}\cdot D_{i,j}. 
 $$
 This shows the desired relation $A_{i, j}\cdot B_{i, j}\cdot C_{i, j}\cdot D_{i, j} \stackrel{3}{\sim}  \mathbf{1}$ by Theorem \ref{thm:fgg}. 
\end{proof}

\begin{notation} \label{notaw3nbr}
For pairwise distinct indices $i, j,k \in \{1,\cdots,n\}$, we set
\begin{enumerate}
	\item[(i). ] $\!\!\!\!\phi_{i, j, k} =  \mathcal{I}_{(\overline{j}, i,  \overline{k}) ; 2} - \mathcal{I}_{(\overline{i}, j) ; 2} 
         - \mathcal{I}_{(i,\overline{k}) ; 2} - \mathcal{I}_{(k,j) ; 2} +  \mathcal{I}_{(i) ; 2} + \mathcal{I}_{(j) ; 2} + \mathcal{I}_{(k) ; 2} - \mu(ji) \mu(ki)$; 
	\item[(ii). ]  
		$\alpha_{i, j} = \mathcal{I}_{(i, \overline{j}) ; 2} + \mathcal{I}_{(\overline{i},j) ; 2}  -\mathcal{I}_{(j,i) ; 2} - \mathcal{I}_{(i) ; 2} - \mathcal{I}_{(j) ; 2}$;
	\item[(iii). ]  
		$\beta_{i, j} = \mathcal{I}_{(\overline{i}, j) ; 2} - \mathcal{I}_{(i) ; 2} - \mathcal{I}_{(j) ; 2}$;
	\item[(iv). ]  
		$\gamma_{i, j} = \mathcal{I}_{(\overline{i},j) ; 2} - \mathcal{I}_{(j,i) ; 2}$. 
\end{enumerate}
\end{notation}

\begin{thm} \label{classiw3nbr}
Let $L$ be an $n$-component welded string link. We have 
$$L \stackrel{3}{\sim} L_1 \cdot L_2,$$
where $L_1=\prod_{i \neq j} Z_{i, j}^{\mu_L(ij)}$  with product taken according to the lexicographic order on  $(i,j)$, and 
$$L_2 = \prod_{1 \leq i \leq n} E_i^{\mathcal{I}_{(i) ; 2}(L)}\ \prod_{1 \leq i < j \leq n} A_{i, j}^{\alpha_{i, j}(L)} B_{i, j}^{\beta_{i, j}(L)} C_{i, j}^{\gamma_{i, j}(L)}\ \prod_{\substack{1 \leq i \leq n \\ 1 \leq j < k \leq n \\ i \neq j, i \neq k}} G_{i, j, k}^{\phi_{i, j, k}(L)}.$$
\end{thm}

\begin{corollary} \label{corow3}
The following are equivalent.
\begin{enumerate}
	\item Two welded string links $L$ and $L'$ are $w_3$-equivalent;
	\item For any finite type invariant  $\nu$ of degree at most $2$, we have $\nu(L) = \nu(L')$;
	\item $L$ and $L'$ have same linking numbers  $\mu(ij)$, and same closure invariants 
		$\mathcal{I}_{(i) ; 2}$ for all $i$, $\mathcal{I}_{(j,i) ; 2}$, $\mathcal{I}_{(i, \overline{j}) ; 2}$ and  
			 $\mathcal{I}_{(\overline{i}, j) ; 2}$ for all $i<j$, and 
			 $\mathcal{I}_{(\overline{j}, i, \overline{k}) ; 2}$ for all pairwise distinct $i,j,k$ such that $j < k$.  
      \end{enumerate}
\end{corollary}
\begin{proof}
The fact that $(1)$ implies $(2)$ is given by Proposition \ref{thm:wkFTI}, and since all invariants listed in (3) are degree $\le 2$ invariants (Lemma \ref{Milnortypefini} and Corollary \ref{cor:fticlose}), we have implication $(2)\Rightarrow (3)$. Theorem \ref{classiw3nbr} shows that $(3)$ implies $(1)$.
\end{proof}

\begin{remark}\label{rem:w2}
Theorem \ref{classiw3nbr} also implies that a welded string $L$ is $w_2$-equivalent to the product $\prod_{i \neq j} Z_{i, j}^{\mu_L(ij)}$. As a direct consequence, we have that two welded string links 
are $w_2$-equivalent if and only if they have same welded linking numbers $\mu(ij)$, a result that was first proved in \cite{ABMW3}. 
\end{remark}

\begin{proof}[Proof of Theorem \ref{classiw3nbr}]
We start with an arbitrary $w$-tree presentation of $L$.
By Corollary \ref{lem:exchange}, we have that $L$ is $w_3$-equivalent to $L'_1\cdot L'_2$, 
where $L'_i$ is a product of welded string links, each obtained from $\mathbf{1}$ by surgery along a single $w_i$-tree ($i=1,2$).

By the  involutivity of twists, the Reversal moves (2) and the Isolated arrow move (3), 
we can freely assume that each $w_1$-tree in $L'_1$ is a copy of either $Z_{i, j}$ or $Z_{i, j}^{-1}$. By Remark \ref{identification d'inverse}, we thus have 
 $$L \stackrel{3}{\sim} \prod_{i \neq j} Z_{i, j}^{e_{i, j}}\cdot L'_2,$$
for some coefficients $e_{i, j}\in \mathbb{Z}$. 
Observe that the welded linking numbers are additive under stacking, and recall that they are $w_2$-equivalence invariants by Lemma \ref{Milnortypefini} and Proposition \ref{thm:wkFTI}. 
Hence, applying $\mu(kl)$  (for some $k\neq l$) to this equivalence gives
$$\mu_L(kl) = \sum_{i \neq j} e_{i, j}\times \mu_{Z_{i,j}}(kl).$$
An elementary computation gives that $\mu_{Z_{i,j}}(kl) = \delta_{ik} \delta_{jl}$, hence 
$e_{i, j}=\mu_L(ij)$ for any pair $(i, j)$ of distinct integers.

We now focus on $L'_2$. 
Consider a $w_2$-tree $T$ for $\mathbf{1}$, such that $\mathbf{1}_T$ is a factor of $L'_2$. 
Let us first show that, up to $w_3$-equivalence, $T$ can be assumed to be a copy of $E_i^{\pm 1}$, $A_{i, j}^{\pm 1}$, $B_{i, j}^{\pm 1}$, $C_{i, j}^{\pm 1}$ or $G_{i, j, k}^{\pm 1}$.
Suppose first that all three endpoints of $T$ are on the same component, say component $i$. Then by the Fork move (10), $\mathbf{1}_T$ is nontrivial 
only if the head of $T$ is located between both tails of $T$, and the Reversal moves (2), AS move (9) and Twist relation (11) 
ensure that $T$ is necessarily a copy of either $E_i$ or $E_i^{-1}$. 
In the case where $T$ is attached to exactly two components of $\mathbf{1}$,  say $i$ and $j$, then the same combinatorial arguments give that $T$ 
can be freely assumed to be a copy of $A_{i, j}^{\pm 1}$, $B_{i, j}^{\pm 1}$, $C_{i, j}^{\pm 1}$ or $D_{i, j}^{\pm 1}$. Hence by Proposition \ref{relationABCD}, we can further assume that $T$ is either $A_{i, j}^{\pm 1}$, $B_{i, j}^{\pm 1}$ or $C_{i, j}^{\pm 1}$ with $i<j$.
Finally, in the case where the three endpoints of $T$ lie on pairwise distinct components $i,j,k$, then the same considerations show that $T$ is a copy of $G_{i, j, k}^{\pm 1}$ for pairwise distinct indices $i,j,k$ such that $j<k$. 
Moreover by Corollary \ref{lem:exchange}, any two factors obtained from $\mathbf{1}$ by surgery along a $w_2$-tree, commute up to $w_3$-equivalence. 
Summarizing, we have proved that 
\begin{equation} \label{eqnw3nbr}
L \stackrel{3}{\sim} \underbrace{\prod_{ i \neq j}
Z_{i, j}^{\mu_L(ij)}}_{ = L_1}\ \prod_{1 \leq i \leq n} E_i^{e_i}\ \prod_{1 \leq i < j \leq n} A_{i, j}^{a_{i, j}} B_{i, j}^{b_{i, j}} C_{i, j}^{c_{i, j}}\ \prod_{\substack{1 \leq i \leq n \\ 1 \leq j < k \leq n \\ i \neq j, i \neq k}} G_{i, j, k}^{g_{i, j, k}},
 \end{equation}
for some $g_{i, j, k}$, $a_{i, j}$, $b_{i, j}$, $c_{i, j}$ and $e_i$ in $\mathbb{Z}$.
In what follows, for convenience we shall call \emph{basic factor} any factor appearing in the product (\ref{eqnw3nbr}). 
\\
Consider the invariant $\mathcal{I}_{(i) ; 2}$, which is the normalized Alexander coefficient $\alpha_2$ of the $i$th component.  By \cite[Lem.~6.4]{Meilhan-Yasuhara}, we have that  $\mathcal{I}_{(i) ; 2}(E_i^{\pm 1}) = \pm 1$ and $\mathcal{I}_{(i) ; 2}$ vanishes on any other basic factor.
Recall that $\mathcal{I}_{(i) ; 2}$ is an invariant of $w_3$-equivalence (Theorem \ref{thm:wkFTI}). By the additivity property of Proposition \ref{clÃŽture et additivitÃ©}, evaluating $\mathcal{I}_{(i) ; 2}$ on $(\ref{eqnw3nbr})$ thus gives us that $e_i = \mathcal{I}_{(i) ; 2}(L)$.
\\
Next we evaluate the closure invariants $\mathcal{I}_{(i, \overline{j})}$, $\mathcal{I}_{(\overline{i}, j)}$ and $\mathcal{I}_{(i, j)}$ ($i<j$) on the following basic factors: 
$$\begin{pmatrix}
\smallsetminus & A_{i, j} & B_{i, j} & C_{i, j} & E_i & E_j \\
\mathcal{I}_{(i, \overline{j}) ; 2} : & 1 & 0 & -1 & 1 & 1 \\
\mathcal{I}_{(\overline{i}, j) ; 2} : & 0 & 1 & 0 & 1 & 1 \\
\mathcal{I}_{(j, i) ; 2} : & 0 & 1 & -1 & 1 & 1  \\
\end{pmatrix}$$
Moreover, these three closure invariants vanish on $L_1$, and on all basic factors $G_{i, j, k}$. 
By Theorem \ref{thm:wkFTI} and Proposition \ref{clÃŽture et additivitÃ©}, evaluating these invariants on $(\ref{eqnw3nbr})$ gives: 
\begin{eqnarray*}
	\mathcal{I}_{(i, \overline{j}) ; 2}(L) & = & a_{i, j} - c_{i, j} + e_i + e_j,\\
	\mathcal{I}_{(\overline{i}, j) ; 2}(L)& = & b_{i, j} + e_i + e_j,\\
	\mathcal{I}_{(j,i) ; 2}(L) & = & b_{i, j} - c_{i, j} + e_i + e_j.
\end{eqnarray*}
Consequently, we have
\begin{enumerate}
	\item[ ] $a_{i, j} = \mathcal{I}_{(i, \overline{j}) ; 2}(L) + \mathcal{I}_{(\overline{i},j) ; 2}(L)  -\mathcal{I}_{(j,i) ; 2}(L) - \mathcal{I}_{(i) ; 2}(L) - \mathcal{I}_{(j) ; 2}(L)$;
	\item[ ] $b_{i, j} = \mathcal{I}_{(\overline{i}, j) ; 2}(L) - \mathcal{I}_{(i) ; 2}(L) - \mathcal{I}_{(j) ; 2}(L)$;
	\item[ ] $c_{i, j} =  \mathcal{I}_{(\overline{i},j) ; 2}(L) - \mathcal{I}_{(j,i) ; 2}(L)$.
\end{enumerate}
Finally, the closure invariant $\mathcal{I}_{(\overline{j}, i, k) ; 2}$ takes the following values on basic factors: 
$$\setcounter{MaxMatrixCols}{22}
\arraycolsep=3pt
\begin{pmatrix}
\smallsetminus\! &\! G_{i, j, k}\! &\! A_{i, j}\! &\! B_{i, j}\! &\! C_{i, j}\! &\! A_{i, k}\! &\! B_{i, k}\! &\! C_{i, k}\! &\! A_{j, k}\! &\! B_{j, k}\! &\! C_{j, k}\! &\! E_i\! &\! E_j\! &\! E_k\! &\! \star \\
\mathcal{I}_{(\overline{j}, i, \overline{k}) ; 2} : & 1 & 0 & 1 & 0 & 1 & 0 & -1 & 0 & 1 & -1 & 1 & 1 & 1 & 0
\end{pmatrix}$$
where $\star$ stands for any other basic factor. 
Hence by Theorem \ref{thm:wkFTI} and Proposition \ref{clÃŽture et additivitÃ©}, we have: 
$$\mathcal{I}_{(\overline{j}, i, \overline{k}) ; 2}(L) - \mathcal{I}_{(\overline{j}, i, \overline{k}) ; 2}(L_1) = g_{i, j, k} + b_{i, j} + a_{i, k} - c_{i, k} + b_{j, k} - c_{j, k} + e_i + e_j + e_k. $$
Here however, unlike in the preceding computation, $\mathcal{I}_{(\overline{j}, i, k) ; 2}$ does not vanish on $L_1$.
Indeed, the closure $Cl_{(\overline{j},i,k)}$ of the following diagram 
\begin{center}
\begin{tikzpicture}
	\draw[<-] [very thick] (0,1.5) -- (0,0);
	\draw (-0.25,0.25) node [below] {$i$};
	\draw[<-] [very thick] (1,1.5) -- (1,0);
	\draw (0.75,0.25) node [below] {$j$};
	\draw[<-] (0,1) -- (1,1);
	\draw[<-] [very thick] (2,1.5) -- (2,0);
	\draw (1.75,0.25) node [below] {$k$};
	\draw[<-] (0,0.5) -- (2,0.5);
\end{tikzpicture}
\end{center}
yields the welded long knot $K$ of Figure \ref{fig:nonab}, which satisfies $\alpha_2(K)=1$. 
Based on this observation, a computation shows that (see \cite[Lem.~2.3.18]{Casejuane}): 
$$\mathcal{I}_{(\overline{j}, i, \overline{k}) ; 2}(L_1)=\mu_L(ji)\mu_L(ki).$$ 
It follows that 
\begin{eqnarray*}
	g_{i, j, k} & = & \mathcal{I}_{(\overline{j}, i,  \overline{k}) ; 2}(L)  - \mu_L(ji)\mu_L(ki) 
	- b_{i, j} - a_{i, k} + c_{i, k} - b_{j, k} + c_{j, k} - e_i - e_j - e_k\\
   & =  & \mathcal{I}_{(\overline{j}, i,  \overline{k}) ; 2}(L) - \mathcal{I}_{(\overline{i}, j) ; 2}(L) 
         - \mathcal{I}_{(i,\overline{k}) ; 2}(L) 
         -\mathcal{I}_{(k,j) ; 2}(L) \\
   &   & \,+\, \mathcal{I}_{(i) ; 2}(L) + \mathcal{I}_{(j) ; 2}(L) + \mathcal{I}_{(k) ; 2}(L) - \mu_L(ji) \mu_L(ki).
\end{eqnarray*}
This gives the desired formula.
\end{proof}

A notable observation about Corollary \ref{corow3} is that degree $2$ finite type invariants of welded string links are generated by closure invariants -- while degree $1$ invariants are generated by the welded linking numbers. 
This means that any other degree $2$ invariant can be expressed as a linear combination of (products of) such invariants, and Theorem \ref{classiw3nbr} can be used effectively to make this explicit. 
The next result gives such a formula for length $3$ welded Milnor invariants. 
\begin{proposition} \label{mu_123}
Let $i,j,k$ be pairwise distinct indices. We have 
\begin{eqnarray*}
\mu(ijk) & = & \mu(ji)\mu(ik) - \mu(ij)\mu(jk) - \mu(ik) \mu(jk) \\
& &
 +\, \mathcal{I}_{(\overline{i}, k,  \overline{j}) ; 2} - \mathcal{I}_{(\overline{k}, i) ; 2} 
         - \mathcal{I}_{(k,\overline{j}) ; 2} - \mathcal{I}_{(j,i) ; 2} + \mathcal{I}_{(i) ; 2} + \mathcal{I}_{(j) ; 2} +  \mathcal{I}_{(k) ; 2} .
\end{eqnarray*}
\end{proposition}
\begin{proof}
It suffices to prove the result for $\mu_L(123)$, where $L$ is a $3$-component welded string link. 
Since $\mu(123)$ is a degree $2$ invariant, and using the additivity property of \cite[Lem.~6.11]{Meilhan-Yasuhara}, evaluating on the $w_3$-equivalence class representative of Theorem \ref{classiw3nbr} gives 
$$\mu_L(123) =  \mu_{L_1}(123)+ \mu_{L_2}(123) =  \mu_{L_1}(123) + \phi_{3,1,2}(L)\times \underbrace{\mu_{G_{3,1,2}}(123)}_{=1}.\\[-0.3cm]$$
A direct computation gives 
 $$\mu_{L_1}(123) = \mu_L(21)\mu_L(13) - \mu_L(12)\mu_L(23),$$
and the desired formula then follows from the definition of the invariant $\phi_{3,1,2}$. 
\end{proof}

\begin{remark}
The main result of this section, Corollary \ref{corow3}, should be seen as a welded analogue of \cite[Thm.~4.23]{jbjktr}, as restated in \cite[Thm.~2.2]{MY3}, 
in the classical case. There, it is shown that two classical string links are $C_3$-equivalent if and only if  they have same Vassiliev invariants of degree $<3$, 
which is equivalent to having same linking numbers, Milnor's triple linking numbers, Casson knot invariants of each component, and a closure-type invariant, 
namely the Casson invariant of the closure $Cl_{1,\overline{2}}$. Observe that the welded case of Corollary \ref{corow3} involves a significantly greater number 
of invariants. Now, Proposition 2.10 of  \cite{jbjktr} expresses the classical triple linking number in terms of closure invariants\footnote{
Proposition \ref{mu_123} is to be compared to \cite[Prop.~2.10]{jbjktr} (this formula  was later widely generalized in \cite{MY_GT}).}. 
This shows that, in the classical setting as well, degree $2$ invariants are generated by closure invariants. It follows from the results of \cite{MY3} and \cite{MY_GT} that this remains true at least up to degree $5$; see Remark \ref{rmqpourplustard} for the welded case. 
\end{remark}

\subsection{Towards a $w_4$-classification of welded string links}\label{sec:A}

As indicated in the Introduction, the characterization of degree $3$ finite type invariants of  welded string links and ribbon tubes, was investigated in detail \cite{Casejuane}, but a complete result is not known. 
In this final section, we outline the $2$-component case, referring the reader to \cite{Casejuane} for the general case. We expect that this exploratory section will lay the ground for future works. 

Consider the welded long knots $F$ and $F'$, and the $2$-component welded string links $A, B, C, D$ and $TO_i, OT_i$ ($i=1,2,3,4$) shown below. 
\begin{center}
\begin{tikzpicture}
	\draw[<-] [very thick] (0,2) -- (0,0);
	\draw[<-] (0,1.33) -- (0.33,1);
	\draw (0.33,1) -- (0.66,1);
	\draw (0.33,1) -- (0,0.66);
	\draw (0,0.33) arc(-90:90:0.66);
	\draw (0,0) node [below] {$F$};
\end{tikzpicture}
$\quad\,\,\,\,$
\begin{tikzpicture}
	\draw[<-] [very thick] (0,2) -- (0,0);
	\draw[<-] (0,1.5) -- (1,1.5);
	\draw (0.5,1.5) -- (0.5,0.5);
	\draw (0,0.5) -- (1,0.5);
	\draw[<-] [very thick] (1,2) -- (1,0);
	\draw (0.5,0) node [below] {$A$};
\end{tikzpicture}
$\quad\,\,\,\,$
\begin{tikzpicture}
	\draw[<-] [very thick] (0,2) -- (0,0);
	\draw (0,1.5) -- (1,1.5);
	\draw (0.5,1.5) -- (0.5,0.5);
	\draw[<-] (0,0.5) -- (1,0.5);
	\draw[<-] [very thick] (1,2) -- (1,0);
	\draw (0.5,0) node [below] {$B$};
\end{tikzpicture}
$\quad\,\,\,\,$
\begin{tikzpicture}
	\draw[<-] [very thick] (0,2.5) -- (0,0);
	\draw[->] (0,0.5) arc(270:450:0.75);
	\draw (0,1.5) -- (0.7,1.5);
	\draw (0.7,1) -- (1,1);
	\draw[<-] [very thick] (1,2.5) -- (1,0);
	\draw (0.5,0) node [below] {$TO_1$};
\end{tikzpicture}
$\quad\,\,\,\,$
\begin{tikzpicture}	\draw[<-] [very thick] (3,2.5) -- (3,0);
	\draw (4,0.5) arc(270:90:0.75);
	\draw[->] (4,2) arc(270:450:0.125);
	\draw (3.3,1.5) -- (4,1.5);
	\draw (3,1) -- (3.3,1);
	\draw[<-] [very thick] (4,2.5) -- (4,0);
	\draw (3.5,0) node [below] {$OT_1$};
\end{tikzpicture}
$\quad\,\,\,\,$
\begin{tikzpicture}
	\draw[<-] [very thick] (0,2.5) -- (0,0);
	\draw (0,0.5) arc(270:450:0.75);
	\draw[<-] (0,1.5) -- (0.7,1.5);
	\draw (0.7,1) -- (1,1);
	\draw[<-] [very thick] (1,2.5) -- (1,0);
	\draw (0.5,0) node [below] {$TO_2$};
\end{tikzpicture}
$\quad\,\,\,\,$
\begin{tikzpicture}	
        \draw[<-] [very thick] (3,2.5) -- (3,0);
	\draw (4,0.5) arc(270:90:0.75);
	\draw (3.3,1.5) -- (4,1.5);
	\draw[->] (4,1.5) arc(270:450:0.125);
	\draw (3,1) -- (3.3,1);
	\draw[<-] [very thick] (4,2.5) -- (4,0);
	\draw (3.5,0) node [below] {$OT_2$};
\end{tikzpicture}
\end{center}
\begin{center}
\begin{tikzpicture}
	\draw[<-] [very thick] (0,2) -- (0,0);
	\draw (0,1.33) -- (0.33,1);
	\draw (0.33,1) -- (0.66,1);
	\draw[->] (0.33,1) -- (0,0.66);
	\draw (0,0.33) arc(-90:90:0.66);
	\draw (0,0) node [below] {$F'$};
\end{tikzpicture}
$\quad\,\,\,\,$
\begin{tikzpicture}	\draw[<-] [very thick] (3,2) -- (3,0);
	\draw (3,1.5) -- (4,1.5);
	\draw[->] (4,1.5) arc(270:450:0.125);
	\draw (3.5,1.5) -- (3.5,0.5);
	\draw (3,0.5) -- (4,0.5);
	\draw[<-] [very thick] (4,2) -- (4,0);
	\draw (3.5,0) node [below] {$C$};
	\end{tikzpicture}
$\quad\,\,\,\,$
\begin{tikzpicture}	\draw[<-] [very thick] (3,2) -- (3,0);
	\draw (3,1.5) -- (4,1.5);
	\draw (3.5,1.5) -- (3.5,0.5);
	\draw (3,0.5) -- (4,0.5);
	\draw[->] (4,0.5) arc(450:270:0.125);
	\draw[<-] [very thick] (4,2) -- (4,0);
	\draw (3.5,0) node [below] {$D$};
\end{tikzpicture}
$\quad\,\,\,\,$
\begin{tikzpicture}
	\draw[<-] [very thick] (0,2.5) -- (0,0);
	\draw (0,0.5) arc(270:450:0.75);
	\draw (0.7,1.5) -- (1,1.5);
	\draw[<-] (0,1) -- (0.7,1);
	\draw[<-] [very thick] (1,2.5) -- (1,0);
	\draw (0.5,0) node [below] {$TO_3$};
\end{tikzpicture}
$\quad\,\,\,\,$
\begin{tikzpicture}	\draw[<-] [very thick] (3,2.5) -- (3,0);
	\draw (4,0.5) arc(270:90:0.75);
	\draw (3,1.5) -- (3.3,1.5);
	\draw[->] (4,1) arc(270:450:0.125);
	\draw (3.3,1) -- (4,1);
	\draw[<-] [very thick] (4,2.5) -- (4,0);
	\draw (3.5,0) node [below] {$OT_3$};
\end{tikzpicture}
$\quad\,\,\,\,$
\begin{tikzpicture}
	\draw[<-] [very thick] (0,2.5) -- (0,0);
	\draw[<-] (0,0.5) arc(270:450:0.75);
	\draw (0.7,1.5) -- (1,1.5);
	\draw (0,1) -- (0.7,1);
	\draw[<-] [very thick] (1,2.5) -- (1,0);
	\draw (0.5,0) node [below] {$TO_4$};
\end{tikzpicture}
$\quad\,\,\,\,$
\begin{tikzpicture}	\draw[<-] [very thick] (3,2.5) -- (3,0);
	\draw (4,0.5) arc(270:90:0.75);
	\draw[->] (4,0.5) arc(270:450:0.125);
	\draw (3,1.5) -- (3.3,1.5);
	\draw (3.3,1) -- (4,1);
	\draw[<-] [very thick] (4,2.5) -- (4,0);
	\draw (3.5,0) node [below] {$OT_4$};
\end{tikzpicture}
\end{center}

For $i=1,2$, we denote by $F_i$ the $2$-component welded string link obtained from the trivial one by inserting a copy of $F$ on the $i$th component. 
We also denote with a superscript $-1$ the welded string links obtained from the above ones by inserting a $\bullet$ near the head, which by Remark \ref{identification d'inverse} defines the inverse up to $w_4$-equivalence. 

We saw in Section \ref{sec:SL3} that the abelian group $wSL(2)_3$ is generated by the welded string links $Z_{1,2}$, $Z_{2,1}$, $E_1$, $E_2$, $A_{1,2}$, $B_{1,2}$ and $C_{1,2}$. 
In order to capture the next degree case, it hence suffices to understand the set $\overline{wSL(2)}_4$ of $w_4$-equivalence classes of $2$-component welded string links that are $w_3$-equivalent to $\mathbf{1}$. Note that Corollary  \ref{lem:exchange} implies that $\overline{wSL(2)}_4$ is actually an abelian group. 
\begin{proposition}\label{relations w_4}
$\overline{wSL(2)}_4$ is generated by $F_1$, $F_2$, $A$, $B$, $C$, $D$, $TO_1$ and $OT_1$. 
\end{proposition}

We will need the following technical result to prove Proposition \ref{relations w_4}. 
\begin{claim} \label{inversion1}
Let $T_1$ and $T_2$ be two $w_k$-trees for $\mathbf{1}$, which are identical except in a disk where they differ as shown below.
$$ \textrm{\includegraphics{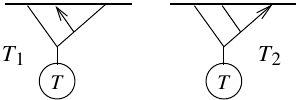}} $$
 We have 
$$\mathbf{1}_{T_1}\cdot \mathbf{1}_{T_2} \stackrel{{k+1}}{\sim} \mathbf{1}.$$ 
\end{claim}
\begin{proof} 
Let us start with the union $A\cup V$ of $w$-trees shown on the left below.  
Exchanging the tail of $A$ with the head of $V$ by move (8), one can isolate and delete $A$ by move (4). The exchange move (8) introduces a
$w_k$-tree, which can be isolated up to $w_{k+1}$-equivalence by Corollary \ref{lem:exchange}. 
More precisely, we obtain: 
$$ \textrm{\includegraphics{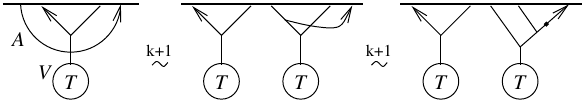}},$$
where the second equivalence is obtained by exchanging the tails of the $w_k$-tree by move (3), followed by the AS move (9) and the Twist relation (11). 

Now let us return to the the union $A\cup V$, and exchange now the head of $A$ with the adjacent tail of $V$ using  the exchange relation (12). This introduces a $w_k$-tree, which can be isolated by Corollary \ref{lem:exchange} as follows: 
$$ \textrm{\includegraphics{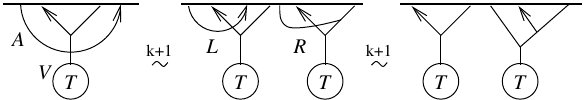}}$$
As we show below, the resulting union $L\cup R$ of $w$-trees satisfies the second equivalence above, and multiplying by the inverse of $V$ in $wSL(n)_{k+1}$ then gives the desired equivalence.

Let us turn to the union $L$ of $w$-trees. 
Exchanging both heads by move (7), we can delete the resulting $w_1$-tree by move (4). As before, using Corollary  \ref{lem:exchange} we then obtain the equivalence on the left-hand side below. 
$$ \textrm{\includegraphics{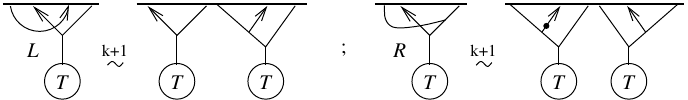}}$$
We now focus on the $w_k$-tree $R$. Using the IHX relation (13) and the AS move (9),
we have the equivalence shown on the right-hand side. Combining these equivalences for $L$ and $R$ indeed provides the desired equivalence, and the proof is complete. 
\end{proof}

\begin{proof}[Proof of Proposition \ref{relations w_4}]
Let $T$ be a $w_3$-tree for the $2$-component welded string link $\mathbf{1}$. 
Similar combinatorial considerations as in the proofs of Theorem  \ref{classiw3nbr} show that, up to $w_4$-equivalence, $T$ can be freely assumed to be a copy of one of the $w_3$-trees listed at the beginning of this section, or its inverse. 
Now, we have the following relations:  
\begin{eqnarray}
	& F' \stackrel{4}{\sim} F^{-1}&  \\        
	TO_1 \stackrel{4}{\sim} TO_2^{-1} & \quad \textrm{ and }\quad &
	TO_3 \stackrel{4}{\sim} TO_4^{-1}\\
	OT_1 \stackrel{4}{\sim} OT_2^{-1} & \quad \textrm{ and }\quad & 
	OT_3 \stackrel{4}{\sim} OT_4^{-1}\\
	A\cdot  OT_3 \stackrel{4}{\sim} B\cdot  OT_2 & \quad \textrm{ and }\quad& 
	C\cdot  TO_3 \stackrel{4}{\sim} D\cdot  TO_2.
\end{eqnarray}
Indeed, using Claim \ref{inversion1}, with $k=3$, we immediately obtain relations (6.i) for $i=1,2,3$.
Let us now prove (6.4): we only show the relation on the left-hand side, since the second relation is proved by the exact same argument. 
The strategy is very similar to the previous proofs, so we only outline the successive operations needed. 
Consider a union $U$ of a $w_1$ and $w_2$-tree, as shown on the left-hand side below. 
Applying the Head/Tail exchange relation (12) at the bottom of component $2$, we obtain the following equivalence: 
\begin{center}
\begin{tikzpicture}
	\draw[<-] [very thick] (0,2) -- (0,0);
	\draw[<-] (0,1) -- (0.5,1);
	\draw (0.5,1) -- (1,1.5);
	\draw (0.5,1) -- (1,0.5);
	\draw[->] (0,0.5) -- (1,1);
	\draw[<-] [very thick] (1,2) -- (1,0);
	
	\draw (1.5,1) node {$\stackrel{4}{\sim}$};
	\draw[<-] [very thick] (2,2) -- (2,0);
	\draw[<-] (2,1.25) -- (2.5,1.25);
	\draw (2.5,1.25) -- (3,1.75);
	\draw (2.5,1.25) -- (3,0.75);
	\draw[->] (2,0.3) -- (3,0.3);
	\draw[<-] [very thick] (3,2) -- (3,0);
	\draw (3.5,1) node {$\, \cdot\ A.$};
\end{tikzpicture}
\end{center}
On the other hand, exchanging the head and tail on component $1$ using relation (12), followed by the IHX relation (13),  gives the first equivalence below:  
\begin{center}
\begin{tikzpicture}
	\draw[<-] [very thick] (0,2) -- (0,0);
	\draw[<-] (0,1) -- (0.5,1);
	\draw (0.5,1) -- (1,1.5);
	\draw (0.5,1) -- (1,0.5);
	\draw[->] (0,0.5) -- (1,1);
	\draw[<-] [very thick] (1,2) -- (1,0);
	
	\draw (1.5,1) node {$\stackrel{4}{\sim}$};
	\draw[<-] [very thick] (2,2) -- (2,0);
	\draw[<-] (2,1) -- (2.5,1);
	\draw (2.5,1) -- (3,1.5);
	\draw (2.5,1) -- (3,0.5);
	\draw[->] (2,1.5) -- (3,1);
	\draw[<-] [very thick] (3,2) -- (3,0);
	\draw (4,1) node {$\, \cdot\ OT_2\ OT_3^{-1}$};
	
	\draw (5.5,1) node {$\stackrel{4}{\sim}$};
	\draw[<-] [very thick] (6,2) -- (6,0);
	\draw[->] (6,1.5) -- (7,1.5);
	\draw[<-] (6,0.75) -- (6.5,0.75);
	\draw (6.5,0.75) -- (7,1.25);
	\draw (6.5,0.75) -- (7,0.25);
	\draw[<-] [very thick] (7,2) -- (7,0);
	\draw (8.5,1) node {$\, \cdot\ B\ OT_2\ OT_3^{-1}$};
\end{tikzpicture}
\end{center}
The second equivalence is then obtained by using relation (12) at the top of component $2$. 
This proves that $A \stackrel{4}{\sim} B\ OT_2\ OT_3^{-1}$. 
\end{proof}

However, a complete $w_4$-equivalence classification result is at this point only accessible modulo the following: 
\begin{conjecture} \label{Conjecture}
 $$A\cdot D \stackrel{4}{\sim} B\cdot C.$$
\end{conjecture}
\begin{remark}
This conjectured relation can be seen as an analogue of relation (\ref{relationABCD}) up to $w_4$-equivalence. 
None of the invariants involved in this paper can distinguish $A\cdot D$ from $B\cdot C$. In particular, Milnor invariants $\mu(1221)$ and $\mu(2112)$ cannot resolve Conjecture \ref{Conjecture}. 
\end{remark}Ê

It is however still interesting to study welded string links up to $w_4$-equivalence modulo this conjecture, as discussed in Remark \ref{rmqpourplustard}.   

\begin{notation} \label{notaw42br}
We set the following $2$-component welded string link invariants:  
\begin{enumerate}
	\item[(i). ] $\gamma_1 = - \mathcal{I}_{(1, 2) ; 3} + \mathcal{I}_{(1) ; 3} + \mathcal{I}_{(2) ; 3} + \mu(1121)$.
	\item[(ii). ]  $\gamma_2 = \mathcal{I}_{(\overline{2}, 1) ; 3} - \mathcal{I}_{(1) ; 3} - \mathcal{I}_{(2) ; 3} + \mathcal{I}_{(1) ; 2} + \mathcal{I}_{(2) ; 2}$.
	\item[(iii). ]  $\gamma_3 = - \mathcal{I}_{(2, 1) ; 3} + \mathcal{I}_{(\overline{2}, 1) ; 3} + \mu(2212)$.
\end{enumerate}
\end{notation}

\begin{thm} \label{classiw42br}
Let $L$ be a $2$-component welded string link. Assuming Conjecture  \ref{Conjecture}, we have
$$L \stackrel{4}{\sim} L_1\cdot L_2  \cdot L_3,$$
where 
$$ 
L_1\cdot L_2 = Z_{1, 2}^{\mu_L(12)} Z_{2, 1}^{\mu_{L}(21)}\cdot E_1^{\mathcal{I}_{(1) ; 2}(L)}\cdot  E_2^{\mathcal{I}_{(2) ; 2}(L)}\cdot A_{1,2}^{\alpha_{1,2}(L)}\cdot B_{1,2}^{\beta_{1,2}(L)}\cdot C_{1,2}^{\gamma_{1,2}(L)}$$ 
as in Theorem \ref{classiw3nbr}, and where $L_3$ is given by 
\begin{align*}
	(F_1)^{\mathcal{I}_{(1) ; 3}(L) - \mathcal{I}_{(1) ; 3}(L_{12})}\cdot (F_2)^{\mathcal{I}_{(2) ; 3}(L) - \mathcal{I}_{(2) ; 3}(L_{12})}\cdot 
	(A)^{\gamma_1(L) - \gamma_1(L_{12})}\cdot (B)^{\gamma_2(L) - \gamma_2(L_{12})}\\
	 \cdot (C)^{\gamma_3(L) - \gamma_3(L_{12})}\cdot (TO_1)^{\mu_{L}(1121) - \mu_{L_{12}}(1121)}\cdot (OT_1)^{\mu_L(2212) - \mu_{L_{12}}(2212)}.
\end{align*}
\end{thm}

As before, we immediately deduce the following characterization result. 
\begin{corollary} \label{corow4}
Assuming Conjecture \ref{Conjecture}, the following are equivalent.
\begin{enumerate}
	\item Two $2$-component welded string links $L$ and $L'$ are $w_4$-equivalent;
	\item For any finite type invariant  $\nu$ of degree at most $3$, we have $\nu(L) = \nu(L')$;
	\item $L$ and $L'$ have same Milnor invariants $\mu(12)$, $\mu(21)$, $\mu(1121)$ and $\mu(2212)$, 
	and same closure invariants $\mathcal{I}_{(1, 2) ; 3}$, $\mathcal{I}_{(2, 1) ; 3}$, $\mathcal{I}_{(\overline{2}, 1) ; 3}$, $\mathcal{I}_{(1) ; 3}$, $\mathcal{I}_{(2) ; 3}$, $\mathcal{I}_{(2, 1) ; 2}$, $\mathcal{I}_{(1, \overline{2}) ; 2}$, $\mathcal{I}_{(\overline{1}, 2) ; 2}$, $\mathcal{I}_{(1) ; 2}$, $\mathcal{I}_{(2) ; 2}$.
\end{enumerate}
\end{corollary}

\begin{proof}[Proof of Theorem \ref{classiw42br}]
By Corollary \ref{lem:exchange}, $L$  is $w_4$-equivalent to a product of terms, each obtained from $\mathbf{1}$ by surgery along a single $w_i$-tree ($i\le 3$), ordered by their degree.
Following the proofs of Theorem \ref{classiw3nbr}, we can further assume that 
$L \stackrel{4}{\sim} L_1 \cdot L_2 \cdot \tilde{L_3}$, 
where $L_1$ and $L_2$ are as given in the statement and where $\tilde{L_3}$ is a product of terms, each obtained from $\mathbf{1}$ by surgery along a single $w_3$-tree. 
Proposition \ref{relations w_4}, then gives us that  
$$
 L \stackrel{4}{\sim} L_1\cdot L_2\cdot A^a B^b C^c (TO_1)^t (OT_1)^u (F_1)^{e_1} (F_2)^{e_2},
$$
for some coefficients $a, b, c, t, u, e_1$ and $e_2$ in $\mathbb{Z}$, that we must now determine.   
We have the following evaluations of our invariants: 
$$ \setcounter{MaxMatrixCols}{12}
\begin{pmatrix}
\smallsetminus & A & B & C & TO_1 &  OT_1 & F_1 & F_2 \\
\mathcal{I}_{(\overline{2}, 1) ; 3} : & 0 & 1 & 0 &  0 & 0 & 1 & 1 \\
\mathcal{I}_{(2, 1) ; 3} :            & 0 & 1 & -1 & 0 & 1 & 1 & 1 \\
\mathcal{I}_{(1, 2) ; 3} :            & -1 & 0 & 0 & 1 & 0 & 1 & 1 \\
\mu(1211) :                          & 0 & 0 & 0  & -2 & 0 & 0 & 0 \\
\mu(2122) :                          & 0 & 0 & 0  & 0 & -2 & 0 & 0 \\
\mathcal{I}_{(1) ; 3} : & 0 & 0 & 0 &  0 & 0 & 1 & 0 \\
\mathcal{I}_{(2) ; 3} :            & 0 & 0 & 0 & 0 & 0 & 0 & 1 \\
\end{pmatrix}$$
Note that this matrix has rank $7$. Thanks to the additivity properties of closure invariants (Proposition \ref{clÃŽture et additivitÃ©}) 
and of welded Milnor invariants (\cite[Lem.~6.11]{Meilhan-Yasuhara}), determining the above coefficients then essentially amounts to computing the inverse matrix. 
(Note that here, unlike in Theorem \ref{classiw3nbr}, we do not make explicit the evaluations of our invariants on the degree $\le 2$ part $L_{1}\cdot L_{2}$, 
as it is not necessary for deriving Corollary \ref{corow4}.) 
Details can be found in \cite{Casejuane} and are left to the reader.
\end{proof}

\begin{remark} \label{rmqpourplustard}
Corollary \ref{corow4} suggests that, unlike in the degree $2$ case, one cannot generate the space of degree $3$ finite type invariants of welded string links by only closure invariants. 
As a matter of fact, further computations show that one cannot replace the classifying invariants  $\mu(1121)$ and $\mu(2212)$ by any combination of the closure invariants $\mathcal{I}_{R;3}$ with $R$ a list of length $\le 2$. 
\end{remark}

\section{Application to ribbon knotted surfaces} \label{sec:ribbon}

As mentioned in the introduction, one of the main features of welded theory is that it can serve as a tool for the study of certain surfaces in $4$-space, called ribbon surfaces. 
As a matter of fact, all definitions and main results given in this paper for welded string links, do translate naturally to topological results. This relies on the so-called \emph{Tube map}, due by Satoh \cite{Satoh2000}, which is a surjective map from welded knotted objects to ribbon knotted surfaces. In the context of this paper, this map is a surjective  monoid homomorphism
 $$\textrm{Tube: }wSL(n)\longrightarrow \textrm{rT}(n),$$
where  $\textrm{rT}(n)$ is the monoid of $n$-component \emph{ribbon tubes}, up to isotopy fixing the boundary, with composition given by stacking. 
We shall not recall the precise definition of ribbon tubes here, but rather refer the reader to \cite[\S~2.1]{Audoux-Bellingeri-Meilhan-Wagner2018} for a detailed treatment. Likewise, we refer to \cite{Satoh2000} or \cite[\S~3.3]{Audoux-Bellingeri-Meilhan-Wagner2018} for  the definition of the Tube map. 

A key property of ribbon knotted surfaces is that they admit a finite type invariants theory, which was developed in \cite{Habiro-Kanenobu-Shima, KS}. The definition is strictly the same as Definition \ref{def:fti}, with the role of the virtualization move now played by the \emph{crossing change at crossing circles}:
$$  \dessin{1.8cm}{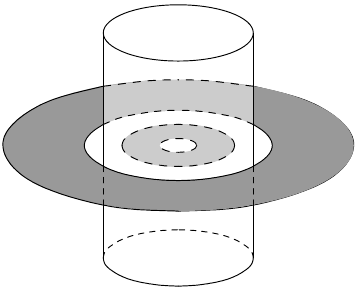}\,\, \longleftrightarrow\,\,
      \dessin{1.8cm}{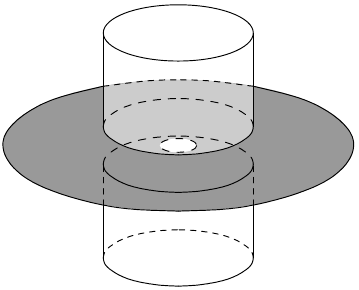}
$$
\noindent By \cite{Yajima}, any ribbon knotted surface admits a diagram where the only crossings are along 'crossing circles' of double points, as shown in the above figure, and the local move swaps the over/under information at this circle.  As observed in \cite{Audoux-Bellingeri-Meilhan-Wagner2018}, if two welded string links differ by a virtualization move, then their images by the Tube map differ by a crossing change at a crossing circle.
By definition, if $v$ is a welded string link invariant, that extends to an invariant $v^{(4)}$ of ribbon tubes 
in the sense that $v^{(4)}(\textrm{Tube}(L)) = v(L)$ for any $L\in wSL(n)$, and if $v$ is a finite type invariant of degree  $k$, then so is $v^{(4)}$. 
Note that this observation applies to all closure invariants and all Milnor invariants, owing to the fact that the Tube map induces an isomorphism from the welded group of $L$ to the fundamental group of the exterior of Tube$(L)$, which preserves peripheral elements (meridians and preferred longitudes) from which these invariants are extracted, see \cite[Sec.~2.2.1]{Audoux-Bellingeri-Meilhan-Wagner2018}. 

Finally, recall that Watanabe introduced in \cite{Watanabe} the \emph{$RC_k$-equivalence} for ribbon knotted surfaces, and showed that two $RC_k$-equivalent ribbon surfaces cannot be distinguished by finite type invariants of degree $< k$. We shall not recall here the definition of $RC_k$-equivalence, but only note the following fact (see \cite{Meilhan-Yasuhara}): two welded string links that are $w_k$-equivalent, have $RC_k$-equivalent images by the Tube map. 

Combining the above facts on the Tube map with the results of this paper, has several concrete consequences for ribbon tubes. 
Using the surjectivity and additivity of the Tube map, we have the following from the results of Section \ref{sec:grooop}.
\begin{corollary}
The set $\textrm{rT}(n) _k$ of $RC_k$-equivalence classes of $\textrm{rT}(n)$, is a finitely generated group. This group is abelian if and only if $k=1$ or $n=1$. 
\end{corollary} 
The characterization of degree $<3$ finite type invariants of ribbon tubes, stated in Theorem \ref{thmain}, likewise follows immediately from Corollary \ref{corow3}. 
 
Parallel to Conjecture \ref{conjek}, this result in low degree raises the following.
\begin{conjecture}
Two ribbon tubes are $RC_k$-equivalents if and only if they cannot be distinguished by finite type invariants of degree $<k$. 
\end{conjecture}
Of course, we also have analogues of the normal form result, Theorem \ref{classiw3nbr}, and of the results of subsection \ref{sec:A}, that we shall not state here explicitly.

\begin{acknowledgments}
The authors are indebted to the referee for numerous useful comments. 
This work is partially supported by the project AlMaRe (ANR-19-CE40-0001-01) of the ANR. 
\end{acknowledgments}

\bibliographystyle{plain}
\bibliography{Bibliographie}
\end{document}